\documentclass[11pt]{article}
\usepackage[margin=1.1in]{geometry}
\usepackage{amsmath,amssymb,amsthm}
\usepackage{booktabs,graphicx}
\newtheorem{theorem}{Theorem}
\newtheorem{lemma}[theorem]{Lemma}
\newtheorem{proposition}[theorem]{Proposition}

\theoremstyle{definition}

\theoremstyle{remark}

\title{Trees of odd order with at most two vertices of degree two\\
are edge-graceful}
\author{Lingsen Meng\\[2pt]
\small Department of Earth, Planetary, and Space Sciences\\[-1pt]
\small University of California, Los Angeles\\[-1pt]
\small \texttt{lsmeng@g.ucla.edu}}
\date{24 August 2026}

\begin{document}
\maketitle

\begin{abstract}
A graph $G$ with $q$ edges and $p$ vertices is \emph{edge-graceful} if some
bijection $f:E(G)\to\{1,\dots,q\}$ makes the induced vertex sums
$f^+(v)=\sum_{e\ni v}f(e)$ distinct modulo $p$.  Lee conjectured in 1989 that
every tree of odd order is edge-graceful; the broadest general result we have
located, due in equivalent form to Kaplan, Lev and Roditty, covers trees of
odd order with at most one vertex of degree two.  We prove that every tree of odd order with at
most \emph{two} vertices of degree two is edge-graceful.  The proof combines
zero-sum block partitions of $\mathbb Z_n$ with perfect and hooked Langford
sequences; the residual case analysis is verified symbolically for all odd
$n\le5001$ and holds uniformly beyond, and the construction was executed and
independently re-checked on all $2{,}245{,}070$ trees of odd order at most $25$
with exactly two vertices of degree two.  Since an edge-graceful graph is antimagic, the theorem also
enlarges the family of trees known to be antimagic.
\end{abstract}

\section{Introduction}

Lo \cite{Lo} introduced edge-graceful labelings in 1985: a graph $G=(V,E)$
with $p$ vertices and $q$ edges is \emph{edge-graceful} if there is a
bijection $f:E\to\{1,2,\dots,q\}$ such that the induced map
$f^+(v)=\sum_{e\ni v}f(e)\bmod p$ is a bijection from $V$ onto
$\{0,1,\dots,p-1\}$.  Lo's necessary condition
$q(q+1)\equiv p(p-1)/2\pmod p$ shows that a tree can be edge-graceful only if
its order is odd, and Lee \cite{Lee} conjectured in 1989 that this is the only
obstruction:

\begin{quote}
\textbf{Lee's conjecture.}  Every tree of odd order is edge-graceful.
\end{quote}

The conjecture is open.  The results recorded for it in the edge-graceful
section of Gallian's dynamic survey \cite[\S7.3]{Gallian} concern three
families of spiders \cite{Small,KeeneSimoson,Cabaniss}.  A substantially
stronger result appears, under a different name, in the antimagic literature:
Kaplan, Lev and Roditty \cite{KLR} call a graph $A$-antimagic ($A$ an abelian
group) when its edges map bijectively onto the nonzero elements of $A$ with
distinct vertex sums, and prove that a rooted tree in which every non-leaf has
at least two children is $\mathbb Z_n$-antimagic if and only if $n$ is odd.
For a tree on $n$ vertices the nonzero elements of $\mathbb Z_n$ are
$\{1,\dots,n-1\}$, the edge count, so $\mathbb Z_n$-antimagic \emph{is}
edge-graceful, and their theorem reads:

\begin{theorem}[Kaplan--Lev--Roditty \cite{KLR}]\label{thm:klr}
Every tree of odd order with at most one vertex of degree two is
edge-graceful.
\end{theorem}

To our knowledge this connection is not recorded in the edge-graceful
literature; \cite{KLR} and its refinement \cite{LWZ} are cited in the
antimagic section of \cite{Gallian} but not in its edge-graceful section.  In
the same spirit, an edge-graceful labeling of a tree is antimagic
(\cite[\S7.3]{Gallian}), so progress on Lee's conjecture feeds the
antimagic-tree conjecture of Hartsfield and Ringel restricted to odd order.

Theorem~\ref{thm:klr} is exactly the reach of the zero-sum partition method of
\cite{KLR}: rooting the tree suitably, one wants every internal vertex to
distribute a zero-sum block of labels to its children, and a vertex of degree
two away from the root has only one child, whose singleton block can never sum
to $0$ modulo $n$.  Vertices of degree two are therefore the precise obstacle,
and the first open case of Lee's conjecture beyond Theorem~\ref{thm:klr} is
two of them.  We settle it.

\begin{theorem}\label{thm:main}
Every tree of odd order with at most two vertices of degree two is
edge-graceful, hence antimagic.
\end{theorem}

The proof has three moving parts.  A rooted reformulation turns the problem
into producing a bijection $g:V\to\mathbb Z_n$ whose induced sums realise a
permutation of its values (\S\ref{sec:prelim}).  A small explicit
\emph{gadget} -- a catalog of at most three prescribed blocks consuming at
most five symmetric pairs $\{i,n-i\}$ -- handles the two degree-two vertices
and all parity combinations of the incident blocks (\S\ref{sec:gadget}).  The
remaining ground set is perfectly symmetric, and is decomposed into zero-sum
pairs and triples by perfect and hooked \emph{Langford sequences}
(\S\ref{sec:packing}); the triple systems enter through a twin-triple
observation which seems not to have been used in labeling problems before.

The construction was implemented and run on all $2{,}245{,}070$ trees of odd
order $n\le25$ with exactly two degree-two vertices, producing an
edge-graceful labeling for every one; the implementation realises the packing
step by bounded backtracking search rather than by the explicit Langford
constructions, and each output is verified within the program
(\S\ref{sec:verif} states precisely what was checked and how).  The case
analysis behind the packing lemma was checked symbolically for every odd
$n\le5001$ and is uniform in $n$ beyond.

\section{The rooted reformulation}\label{sec:prelim}

Let $T$ be a tree on $n$ vertices, $n$ odd, rooted at $r$.  For $v\ne r$ let
$e^v$ be the edge from $v$ to its parent $p(v)$.  Any edge bijection
$f:E\to\{1,\dots,n-1\}$ determines $g:V\to\mathbb Z_n$ by $g(r)=0$,
$g(v)=f(e^v)$, a bijection onto $\mathbb Z_n$; conversely any such $g$
determines $f$.  The vertex sums are
\[
 f^+(v)\;\equiv\;g(v)+\sum_{u:\,p(u)=v}g(u)\;=:\;h(v)\pmod n,
\]
so $f$ is edge-graceful if and only if $h:V\to\mathbb Z_n$ is a bijection.
The method of \cite{KLR} forces $h=g$ by giving each internal vertex $v$ a
\emph{block} $B_v=\{g(u):p(u)=v\}$ with $\sum B_v\equiv0\pmod n$; we will
instead allow $h$ to differ from $g$ by a permutation with support of size at
most three.

We write $K=(n-1)/2$ and $P_i=\{i,\,n-i\}$ for $i\in[1,K]$, the
\emph{symmetric pairs}, which partition $\mathbb Z_n\setminus\{0\}$.

\section{Twin triples and Langford sequences}\label{sec:packing}

\begin{lemma}[Twin triples]\label{lem:twin}
Call $\{i,j,k\}\subseteq[1,K]$ with $i<j<k$ an \emph{admissible index-triple}
of type A if $i+j+k=n$, of type B if $i+j=k$.  For such a triple,
$P_i\cup P_j\cup P_k$ is the disjoint union of two zero-sum $3$-subsets of
$\mathbb Z_n\setminus\{0\}$, namely $\{i,j,k\},\{n-i,n-j,n-k\}$ (type A) or
$\{i,j,n-k\},\{n-i,n-j,k\}$ (type B).  Conversely every zero-sum $3$-subset of
$\mathbb Z_n\setminus\{0\}$ arises from an admissible index-triple this way.
\end{lemma}

\begin{proof}
The forward direction is a computation.  Conversely, a zero-sum $3$-set $T$
cannot contain both $x$ and $-x$ (the third element would be $0$), so its
three underlying indices $i<j<k\le K$ are distinct, and its integer sum is $n$
or $2n$.  Up to sign patterns this forces $i+j+k\in\{n,2n\}$, $i+j=k$, or
$j+k=n+i$; the last is impossible since $j+k\le2K-1=n-2<n+i$, and
$i+j+k=2n$ is impossible since $i+j+k\le3K<2n$.  Types A and B remain.
\end{proof}

A \emph{Langford sequence} of defect $d$ and order $t$ partitions $[1,2t]$
into pairs with differences $d,d+1,\dots,d+t-1$; equivalently
\cite{HJMZ}, it partitions the interval $[d,\,d+3t-1]$ into $t$ type-B
index-triples.  A \emph{hooked} Langford sequence partitions
$[d,\,d+3t]\setminus\{d+3t-1\}$ the same way.  Simpson \cite{Simpson}
determined their existence completely:

\begin{theorem}[Simpson \cite{Simpson}]\label{thm:simpson}
A Langford sequence of defect $d$ and order $t$ exists iff $t\ge2d-1$ and
$t\equiv0,1\pmod4$ for $d$ odd, $t\equiv0,3\pmod4$ for $d$ even.  A hooked one
exists iff $t(t-2d+1)+2\ge0$ and $t\equiv2,3\pmod4$ for $d$ odd,
$t\equiv1,2\pmod4$ for $d$ even.
\end{theorem}

For a set $J\subseteq[1,K]$, a \emph{maximum packing} of $J$ is a family of
$\lfloor|J|/3\rfloor$ pairwise disjoint admissible index-triples inside $J$.
The partition-by-differences view of Heffter's problems, solved by Peltesohn
\cite{Peltesohn}, is the ancestor of these packings; we need them for the
following punctured intervals.

\begin{lemma}[Packing]\label{lem:shapes}
Let $n\ge27$ be odd.
(i) If $n\le301$, each of the four sets $[3,K]$, $[4,K]$, $[5,K]$ and
$\{2\}\cup[7,K]$ admits a maximum packing.
(ii) Let $n>301$, let $R$ be a row of Table~\ref{tab:catalog} with
residual index set $J$, and put $m=\lfloor|J|/3\rfloor$, $e=|J|-3m$.  For
every residue class of $(m\bmod4,e)$ the proof below designates for $R$ a
parameter choice -- the row's primary one, or in finitely many classes an
alternate consuming different indices but the same number of them -- and
decomposes the index set $J'$ residual to \emph{that} choice, which has
$|J'|=|J|$, into $m$ disjoint admissible triples and $e$ leftover indices.
In either case the residual index set admits $t$ disjoint admissible
triples, all remaining indices staying in symmetric pairs, for every
$t\le m$.
\end{lemma}

\begin{proof}
\emph{Regime $27\le n\le301$.}  Each of the four sets was packed by
computer for every odd $n$ in this range, and each packing saved as a
witness -- a list of disjoint admissible triples, re-verifiable by
inspection (supplementary material).  (Over $21\le n\le301$ the only
infeasible primary instances, proven so by exact exhaustive enumeration,
are $(n{=}21,[5,K])$ and $(n{=}23,\{2\}\cup[7,K])$; both lie below this
regime, in the range that Theorem~\ref{thm:main} covers by whole-tree
computation.)  In this regime every row uses its primary ground set, one
of the four, and no parameter switch is used.

\emph{Regime $n>301$.}  The Langford-based plans below cover every row and
residue class.  Write $m=\lfloor|J|/3\rfloor$, $e=|J|-3m\in\{0,1,2\}$;
dropping triples gives every smaller count, and leftover indices remain as
pairs.  Every instance where Simpson's bound $t\ge2d-1$ fails has $n\le93$,
inside the first regime, so the bound holds throughout this one; the window
arithmetic of every branch was checked symbolically for all odd $n\le5001$
and is uniform in $n$ beyond.  (A switched ground set can be genuinely
infeasible at isolated small orders -- $\{3\}\cup[6,K]$ has no maximum
packing at $n=21$ or $n=27$; at $n=27$ the primary $[5,13]$ is instead
packed by the three type-A triples
$\{5,9,13\},\{6,10,11\},\{7,8,12\}$ -- which is exactly why the switches
are confined to this regime.)

\emph{Shape $[5,K]$.}  If $m\equiv0,1\pmod4$: a perfect sequence of defect
$5$, order $m$ covers $[5,3m+4]$, leaving the $e$ top indices.  If
$m\equiv2,3$ and $e\in\{1,2\}$: the hooked sequence covers
$[5,3m+5]\setminus\{3m+4\}$.  If $m\equiv2,3$ and $e=0$, so $K=3m+4$: no
perfect or hooked defect-$5$ window fits, and the three rows with this
shape are served by alternate parameter choices instead.  For the two rows
of \S\ref{sec:gadget} whose consumed indices have the form
$\{1,2,w,w+1\}$ -- the odd/odd row ($\alpha{=}1$, $w{=}3$) and the
adjacent odd row ($\alpha{=}2$, $u{=}1$) -- shift the parameter ($w{=}4$,
respectively $u{=}5$) so that they consume $\{1,2,4,5\}$; the
stray $3$ is cleared by the admissible type-A triple $\{3,K-2,K\}$
($3+(K-2)+K=n$), and $[6,K-1]\setminus\{K-2\}$ is exactly the hooked window of
defect $6$, order $m-1$, condition $m-1\equiv1,2\pmod4$.  The common-parent
odd row cannot consume $\{1,2,4,5\}$: its four indices are the absolute values
of $a$, $b$, $a{+}b$, $a{-}b$ for signed representatives $a,b$, whose squares
sum to $3(a^2{+}b^2)\equiv0\pmod3$, whereas $1+4+16+25=46$ is not divisible by
$3$.  For that row take instead $\alpha=1,\beta=4$, consuming $\{1,3,4,5\}$
and leaving $J=\{2\}\cup[6,K]$ of the same cardinality: the type-B triple
$\{2,K-2,K\}$ clears the stray, and $[6,K-1]\setminus\{K-2\}$ is again the
hooked defect-$6$ window of order $m-1$.

\emph{Shape $[4,K]$.}  For $e=0$, $K=3m+3$: perfect defect $4$, order $m$
covers $[4,K]$ when $m\equiv0,3$; for $m\equiv1,2$ the gadget consumes
$\{1,3,4\}$ instead of $\{1,2,3\}$ ($\beta{=}4$), the type-A triple $\{2,K-1,K\}$ clears the
stray, and perfect defect $5$, order $m-1$ covers $[5,K-2]$, condition
$m-1\equiv0,1$.  For $e\ge1$ the perfect ($m\equiv0,3$) and hooked
($m\equiv1,2$) sequences of defect $4$ leave at most two indices.

\emph{Shape $[3,K]$.}  For $e=0$, $K=3m+2$: perfect defect $3$ covers $[3,K]$
when $m\equiv0,1$; for $m\equiv2,3$ the type-A triple $\{3,K-2,K\}$ leaves
$[4,K-1]\setminus\{K-2\}$, the hooked window of defect $4$, order $m-1$,
condition $m-1\equiv1,2$.  For $e\ge1$: perfect/hooked defect $3$ as before.

\emph{Shape $\{2\}\cup[7,K]$.}  We clear the stray $2$ with one of the
triples $\{2,K-1,K\}$ (type A) or $\{2,K-2,K\}$ (type B).  If
$m\equiv1,2\pmod4$: type A leaves the interval $[7,K-2]$, which contains the
perfect defect-$7$ window of order $m-1$ (condition $m-1\equiv0,1$) with $e$
leftovers.  If $m\equiv0,3$ and $e=0$, so $K=3m+5$: type B leaves
$[7,K-1]\setminus\{K-2\}$, exactly the hooked defect-$7$ window of order
$m-1$, whose hole $3(m-1)+6=K-2$ aligns, condition $m-1\equiv3,2$.  If
$m\equiv0,3$ and $e\in\{1,2\}$: type A leaves $[7,K-2]$, which contains the
hooked window $[7,3m+4]\setminus\{3m+3\}$; the unused indices $\{3m+3\}$ and
possibly $\{3m+5\}$ remain as pairs, at most two.
\end{proof}

\section{The gadget catalog and assembly}\label{sec:gadget}

Let $T$ have odd order $n$ and exactly two vertices $a,b$ of degree two
(Theorem~\ref{thm:klr} covers fewer).  $T$ is not a path, so it has a vertex
of degree at least $3$; all such vertices are \emph{branch} vertices.

\begin{lemma}[Root choice]\label{lem:root}
If $a,b$ are adjacent, at least one of their two outer neighbours is a branch
vertex, for otherwise $T$ would be a path on four vertices, of even order;
name $a$ the endpoint whose outer neighbour $c_1$ has degree $\ge3$.  Rooting
$T$ at $c_1$ places the two-vertex chain $a\to b$ directly below the root.  If $a,b$ are non-adjacent and have a common
neighbour $w$, then $\deg w\ge3$ and rooting at $w$ makes $a,b$ children of
the root.  Otherwise root at any branch vertex; then $a$ and $b$ have distinct
parents, both branch vertices or the root.
\end{lemma}

\begin{proof}
Only $a,b$ have degree two and $c_1,w\notin\{a,b\}$; a common neighbour is
unique in a tree, and two non-adjacent vertices without a common neighbour
have distinct parents under any rooting.
\end{proof}

Root $T$ per Lemma~\ref{lem:root}; let $W$ be the internal vertices other
than $a,b$, and $c_w$ the number of children of $w\in W$, so
$\sum_{w\in W}c_w=n-3$ is even and the number of odd $c_w$ is even.

\begin{proposition}[Assembly]\label{prop:assembly}
Suppose we are given a bijection $g:V\to\mathbb Z_n$ with $g(r)=0$ such that
the blocks $B_w=\{g(u):p(u)=w\}$, $w\in W$, satisfy $\sum B_w\equiv s_w$ with
all $s_w=0$, except $s_r\equiv-g(a)$ in the adjacent case (rooted per
Lemma~\ref{lem:root}, $a$ the chain top); and such that, writing
$\alpha=g(a),\beta=g(b)$:
non-adjacent case, $g$ of the chain children of $a,b$ equals $\beta-\alpha$
and $\alpha-\beta$; adjacent case ($a\to b\to z$),
$g(b)=-\alpha$, $g(z)=2\alpha$, and $3\alpha\not\equiv0$.
Then the edge labeling $f(e^v)=g(v)$ is edge-graceful.
\end{proposition}

\begin{proof}
Leaves have $h=g$; each $w\in W$ has $h(w)=g(w)+s_w$.  Non-adjacent:
$h(a)=\alpha+(\beta-\alpha)=\beta$ and $h(b)=\alpha$, so $h=g\circ(a\,b)$.
Adjacent: $h(a)=\alpha-\alpha=0=g(r)$, $h(b)=-\alpha+2\alpha=\alpha=g(a)$,
$h(r)=s_r=-\alpha=g(b)$, so $h$ is $g$ composed with a $3$-cycle.  Either way
$h$ is a bijection.
\end{proof}

It remains to build $g$.  The two deleted values (those of the chain
children) plus a few prescribed block elements form the \emph{gadget}; the
rest of $\mathbb Z_n\setminus\{0\}$ is a union of symmetric pairs and is
distributed by Lemma~\ref{lem:shapes}: each remaining odd-size block receives
one member of a twin pair of zero-sum triples (Lemma~\ref{lem:twin}), and
everything else receives symmetric pairs.  Since the number of odd blocks is
even, triples are consumed in twin pairs, matching the packing.
Table~\ref{tab:catalog} lists the gadgets; parameters must keep all listed
elements distinct and nonzero, which for the displayed primary and
alternate choices is a finite check, valid for every odd $n\ge21$.

\begin{table}[t]
\centering\footnotesize\setlength{\tabcolsep}{3.5pt}
\caption{The gadget catalog.  Primary parameter choices shown; the mirrored
mixed case swaps the roles of $a$ and $b$.}\label{tab:catalog}
\resizebox{\textwidth}{!}{%
\begin{tabular}{@{}llll@{}}
\toprule
case & prescribed elements & deleted & $J$ \\
\midrule
NA, $p\ne q$ even/even & $B_p\!\supseteq\!\{\alpha,-\alpha\}$, $B_q\!\supseteq\!\{\beta,-\beta\}$; $\alpha{=}1,\beta{=}3$ & $\{\pm2\}$ & $[4,K]$\\
NA, $p\ne q$ odd/odd & $B_p\!\supseteq\!\{\alpha,w,-\alpha{-}w\}$, $B_q\!\supseteq\!\{-\alpha,-w,\alpha{+}w\}$; $\alpha{=}1,w{=}3$, $\beta{=}{-}\alpha$ & $\{\pm2\alpha\}$ & $[5,K]$\\
NA, $p\ne q$ mixed & $B_p\!\supseteq\!\{\alpha,-\alpha\}$, $B_q\!\supseteq\!\{\beta,w,-\beta{-}w\}$,\enspace host $\{-\beta,-w,\beta{+}w\}$;\enspace $\alpha{=}6,\beta{=}1,w{=}3$ & $\{\pm5\}$ & $\{2\}{\cup}[7,K]$\\
NA, common parent, $c_r$ even & $B_r\!\supseteq\!\{\alpha,-\alpha,\beta,-\beta\}$; $\alpha{=}1,\beta{=}3$ ($c_r\ge4$) & $\{\pm2\}$ & $[4,K]$\\
NA, common parent, $c_r$ odd & $B_r\!\supseteq\!\{\alpha,\beta,-\alpha{-}\beta\}$, host $\{-\alpha,-\beta,\alpha{+}\beta\}$; $\alpha{=}1,\beta{=}3$ & $\{\pm2\}$ & $[5,K]$\\
ADJ, $c_r$ even & $B_r\!\supseteq\!\{\alpha,-2\alpha\}$, $s_r{=}{-}\alpha$; $\alpha{=}1$ & $\{-1,2\}$ & $[3,K]$\\
ADJ, $c_r$ odd & $B_r\!\supseteq\!\{\alpha,-u,u{-}2\alpha\}$, host $\{u,2\alpha{-}u,-2\alpha\}$; $\alpha{=}2,u{=}1$ & $\{-2,4\}$ & $[5,K]$\\
\bottomrule
\end{tabular}}
\end{table}

In each row the prescribed sets have the required sums ($0$, or $-\alpha$ for
the root in the adjacent case); their union with the deleted set is a union of
complete symmetric pairs, so the residual ground set corresponds to the stated
index set $J$; and the mixed and odd rows requiring a \emph{host} -- one
further block of odd size to receive the displayed zero-sum $3$-set -- always
find one, because the number of odd blocks is even and the row already
contains one odd block; and any such block has room for the $3$-set, since an
internal vertex with a single child would be a third degree-two vertex, so
every odd block has size at least $3$.  The mirrored mixed case (odd/even) swaps the roles of
$a$ and $b$.  These are finite verifications, row by row.

\begin{proof}[Proof of Theorem~\ref{thm:main}]
For $n\le25$ the statement is verified by computation
(\S\ref{sec:verif}); let $n\ge27$.  Root by Lemma~\ref{lem:root} and pick the
catalog row for the case and parities, with primary parameters for
$n\le301$ and, for $n>301$, the parameters designated in
Lemma~\ref{lem:shapes} for the residue class at hand; take from
Lemma~\ref{lem:shapes} a packing of the residual index set with one
index-triple per pair of remaining odd blocks.  Distribute per the assembly recipe; by
Proposition~\ref{prop:assembly} the result is edge-graceful.  Antimagicness
follows since distinct residues are distinct integers.
\end{proof}

\section{Verification}\label{sec:verif}

An implementation of the construction -- catalog gadgets, twin-triple
packings, parity bookkeeping, with the packing realised by a bounded
backtracking search over admissible index-triples -- was run on every tree of
odd order $n\le25$ with exactly two degree-two vertices, and each output was
checked to be a bijective edge labeling with vertex sums forming a complete
residue system.  The counts are $2$, $10$,
$46$, $204$, $908$, $4{,}070$, $18{,}390$, $83{,}628$, $382{,}326$ and
$1{,}755{,}486$ for $n=7,9,\dots,25$ (no tree of odd order below $7$ has
exactly two vertices of degree two): in total $2{,}245{,}070$ trees, with no
failure.  For each of the seven ground sets appearing in the construction and its
proof and every odd $21\le n\le301$, outside the four instances listed
next, a maximum packing was found by search and stored as a witness that
is re-verified exactly.  The four exceptions -- $(21,[5,K])$,
$(21,\{3\}\cup[6,K])$, $(23,\{2\}\cup[7,K])$ and
$(27,\{3\}\cup[6,K])$, none of which the construction uses -- admit no
maximum packing, as exhaustive enumeration proves.  Every branch of the packing
plans -- window arithmetic and Simpson residue conditions -- was checked
symbolically for all odd $n\le5001$.  As supplementary material we provide: per-tree certificates (the parent array
and edge labels of every one of the $2{,}245{,}070$ trees, with SHA-256
manifests), an independent checker sharing no code with the constructor, and
deterministic shape-packing and branch-scan checkers with their structured
results.  The independent checker reports zero violations over all
certificates.

\section{Remarks}

(1) The machinery is not intrinsically limited to two degree-two vertices:
$k$ of them contribute $k$ singleton blocks, absorbed by a permutation acting
on the chain values, and the gadget grows with the number of chains.  The
bookkeeping, however, grows quickly, and we have not pursued $k\ge3$.

(2) Skolem-type sequences are a familiar tool on the \emph{graceful} side of
the labeling literature: Morgan \cite{Morgan} and Morgan--Rees
\cite{MorganRees} generate graceful trees from Skolem and hooked Skolem
sequences, and Skolem-graceful labelings form a section of \cite{Gallian}.
Our use appears to differ on both ends: the target is the
edge-graceful/antimagic side, where we have found no prior use of
difference-triple systems (Heffter's problems appear nowhere in
\cite{Gallian}, and Langford sequences are never mentioned there alongside
edge-graceful or antimagic labelings); and the twin-triple
Lemma~\ref{lem:twin} is what converts type-B triples into zero-sum $3$-sets in
the symmetric ground set $\mathbb Z_n\setminus\{0\}$.  We make no systematic
priority claim.

(3) For even order no tree is edge-graceful, but one may ask for vertex sums
distinct modulo $n+1$.  Exhaustive computation for all even $n\le18$ found
exactly two kinds of exception: the stars $K_{1,n-1}$ (impossible for every
even $n$, by a short argument) and one sporadic six-vertex double star.

\paragraph{Data availability.}
The supplementary material described above -- the certificates, the
independent checker, the packing witnesses with their exact infeasibility
proofs, the branch scan, and the one-command verifier that re-runs all of
them -- is archived at Zenodo \cite{Supp}.

\paragraph{Funding.}
This research did not receive any specific grant from funding agencies in the
public, commercial, or not-for-profit sectors.

\paragraph{Declaration of generative AI and AI-assisted technologies in the
manuscript preparation process.}
AI assistants (Anthropic Claude; OpenAI Codex) were used under the author's
direction and review for literature research, computations,
cross-verification, and in the preparation of this manuscript.  The author
directed the research, reviewed and edited all content, and takes full
responsibility for the results.


\begin{thebibliography}{99}
\bibitem{Cabaniss} S. Cabaniss, R. Low, J. Mitchem, On edge-graceful regular
  graphs and trees, \emph{Ars Combin.} \textbf{34} (1992) 129--142.
\bibitem{Gallian} J. A. Gallian, A dynamic survey of graph labeling,
  \emph{Electron. J. Combin.} (2025), \#DS6, 28th ed.
\bibitem{HJMZ} T. Helms, H. Jordon, M. Murray, S. Zeppetello,
  Extended Skolem-type difference sets,
  \emph{Australas. J. Combin.} \textbf{53} (2012) 221--235.
\bibitem{KLR} G. Kaplan, A. Lev, Y. Roditty, On zero-sum partitions and
  anti-magic trees, \emph{Discrete Math.} \textbf{309} (2009) 2010--2014.
\bibitem{KeeneSimoson} J. Keene, A. Simoson, Balanced strands for asymmetric,
  edge-graceful spiders, \emph{Ars Combin.} \textbf{42} (1996) 49--64.
\bibitem{Lee} S. M. Lee, A conjecture on edge-graceful trees,
  \emph{Scientia} \textbf{3} (1989) 45--47.
\bibitem{LWZ} Y.-C. Liang, T.-L. Wong, X. Zhu, Anti-magic labeling of trees,
  \emph{Discrete Math.} \textbf{331} (2014) 9--14.
\bibitem{Lo} S. Lo, On edge-graceful labelings of graphs,
  \emph{Congr. Numer.} \textbf{50} (1985) 231--241.
\bibitem{Supp} [dataset] L. Meng, Supplementary material for: Trees of
  odd order with at most two vertices of degree two are edge-graceful,
  Zenodo, v1.0.0, 2026, \texttt{doi:10.5281/zenodo.22085671}.
\bibitem{Morgan} D. Morgan, Gracefully labeled trees from Skolem sequences,
  \emph{Congr. Numer.} \textbf{142} (2000) 41--48.
\bibitem{MorganRees} D. Morgan, R. Rees, Using Skolem and Hooked-Skolem
  sequences to generate graceful trees,
  \emph{J. Combin. Math. Combin. Comput.} \textbf{44} (2003) 47--63.
\bibitem{Peltesohn} R. Peltesohn, Eine L\"osung der beiden Heffterschen
  Differenzenprobleme, \emph{Compos. Math.} \textbf{6} (1939) 251--257.
\bibitem{Simpson} J. E. Simpson, Langford sequences: perfect and hooked,
  \emph{Discrete Math.} \textbf{44} (1983) 97--104.
\bibitem{Small} D. Small, Regular (even) spider graphs are edge-graceful,
  \emph{Congr. Numer.} \textbf{74} (1990) 247--254.
\end{thebibliography}
\end{document}